\documentclass[a4paper,reqno,11pt]{amsart}
\usepackage[T1]{fontenc}
\usepackage{lmodern}
\usepackage{amsmath,amssymb,amsthm,mathtools}
\usepackage[margin=30mm]{geometry}
\usepackage[hidelinks]{hyperref}
\hypersetup{pdftitle={Graded automorphisms of multiparameter quantum matrix algebras},pdfauthor={Hai Jin}}
\numberwithin{equation}{section}
\newtheorem{theorem}{Theorem}[section]
\newtheorem{proposition}[theorem]{Proposition}
\newtheorem{lemma}[theorem]{Lemma}
\newtheorem{corollary}[theorem]{Corollary}
\theoremstyle{definition}

\newtheorem*{definition*}{Definition}
\newtheorem{remark}[theorem]{Remark}
\newtheorem{example}[theorem]{Example}

\newcommand{\kk}{\Bbbk}
\newcommand{\Aut}{\operatorname{Aut}}
\newcommand{\Autgr}{\operatorname{Aut}_{\mathrm{gr}}}
\newcommand{\GL}{\operatorname{GL}}
\newcommand{\HH}{\mathcal H}
\newcommand{\Ccorner}{\mathfrak C}
\newcommand{\Mat}[2]{\mathcal O_{#1,#2}(M_n(\kk))}

\title[Graded automorphisms of quantum matrices]{Graded automorphisms of multiparameter quantum matrix algebras}
\author{Hai Jin}
\address{School of Mathematical Sciences, Dalian University of Technology, Dalian 116081, China}
\email{haijin@dlut.edu.cn}
\subjclass[2020]{Primary 16W20; Secondary 16T20, 16W50, 16S36}
\keywords{Multiparameter quantum matrix algebra, graded automorphism group, automorphism group, rigidity}
\thanks{This work was supported by the National Natural Science Foundation of China (Grant No. 12601034) and the Fundamental Research Funds for the Central Universities (Grant No. DUT25RC(3)091).}
\date{}

\begin{document}

\begin{abstract}
      We determine the graded automorphism group of the multiparameter quantum matrix algebra $\mathcal O_{\lambda,\mathbf p}(M_n(\kk))$ over an arbitrary field $\kk$ for all $n\ge2$ and $\lambda\ne1$. 
      No other restrictions are assumed on $\mathbf p$ or $\lambda$.
      Every graded automorphism is a torus automorphism, possibly followed by transposition or by simultaneous reversal of the row and column indices. We give the precise parameter conditions for these two involutions. 
      As an application, we also determine the full automorphism group when the multiplicative group generated by the parameters is torsion-free.
\end{abstract}

\maketitle

\section{Introduction}\label{sec:introduction}

Let $\kk$ be an arbitrary field, fix $n\ge2$, and let $\mathbf p=(p_{ij})\in M_n(\kk^*)$ be multiplicatively antisymmetric:
\[
      p_{ii}=1,\quad p_{ij}p_{ji}=1.
\]
For $\lambda\in\kk^*$, the multiparameter quantum matrix algebra $A=\Mat{\lambda}{\mathbf p}$ is the $\kk$-algebra generated by $x_{ij}$, $1\le i,j\le n$, with defining relations \cite[Section~I.2.2, p.~16]{BG2002}
\begin{equation}\label{eq:relations}
	x_{lm}x_{ij}=
	\begin{cases}
		p_{li}p_{jm}x_{ij}x_{lm}+(\lambda-1)p_{li}x_{im}x_{lj},&l>i,\ m>j,\\
		\lambda p_{li}p_{jm}x_{ij}x_{lm},&l>i,\ m\le j,\\
		p_{jm}x_{ij}x_{lm},&l=i,\ m>j.
	\end{cases}
\end{equation}
Each generator has degree one, giving $A$ its standard $\mathbb N_0$-grading. We write $\Aut(A)$ for its automorphism group and $\Autgr(A)$ for the subgroup of automorphisms preserving every homogeneous component.
We assume $\lambda\ne1$ except in Remark~\ref{rem:affine}, which treats quantum affine space.

For $\alpha,\beta\in(\kk^*)^n$, the assignments
\[
      h_{\alpha,\beta}(x_{ij})=\alpha_i\beta_jx_{ij}
\]
define automorphisms of $A$, called torus automorphisms. 
They form a subgroup $\HH\cong(\kk^*)^{2n-1}$ of $\Autgr(A)$; see \cite[Section~2.1]{LL2013} and \cite[Example~6.10]{GY2018}.

For $q\in\kk^*$, the single-parameter quantum matrix algebra $\mathcal O_q(M_n(\kk))$ is recovered by setting $\lambda=q^{-2}$ and $p_{ij}=q$ for all $i>j$. In this case, matrix transposition $\tau(x_{ij})=x_{ji}$ is also a graded automorphism.
When $q$ is not a root of unity, Alev and Chamarie proved that the automorphism group is exactly $\HH\rtimes\langle\tau\rangle$ for $n=2$ \cite{AC1992}.
Launois and Lenagan conjectured that the same description holds for all $n\ge 2$ \cite{LL2007}, and proved it for $n=3$ \cite{LL2013}.
Using the rigidity of quantum tori \cite{Yakimov2014}, Yakimov resolved this conjecture \cite{Yakimov2013}.
Consequently, when $q$ is not a root of unity, one has
\[
      \Aut(\mathcal O_q(M_n(\kk))) = \Autgr (\mathcal O_q(M_n(\kk)))= \HH\rtimes\langle\tau\rangle, \quad n\ge2.
\]
When $q$ is a root of unity, the full automorphism group can be larger than the graded automorphism group. 
Gaddis and Lamkin constructed nongraded automorphisms of $\mathcal O_q(M_n(\kk))$ in this case \cite{GL2024}.

In the multiparameter case, $\tau$ is an automorphism only when the parameters satisfy certain compatibility conditions; see Proposition~\ref{prop:symmetries}.
Gaddis and Lamkin studied centers and automorphisms of PI quantum matrix algebras using noncommutative discriminants \cite{GL2024}. Over an algebraically closed field of characteristic zero, they proved that for $n=2$ and $\lambda^2\ne1$, every automorphism of a PI quantum matrix algebra is graded whenever its center is a polynomial ring.

In this paper, we classify the graded automorphisms of multiparameter quantum matrix algebras $\Mat{\lambda}{\mathbf p}$ over an arbitrary field for all $n\ge2$ and $\lambda\ne1$, without any further restrictions on the parameters.

Throughout the paper, we write
\[
    \bar{i}=n+1-i, \quad 1\le i\le n.
\]
Consider the assignment
\[
    \sigma(x_{ij})=x_{\bar{i},\bar{j}},\quad 1\le i,j\le n.
\]
The main result of this paper is the following classification of graded automorphisms.
\begin{theorem}\label{thm:main}
	Suppose $\lambda\ne1$. Then
	\[
		\Autgr(A)=
		\begin{cases}
			\HH\rtimes\langle\tau\rangle,       &\lambda p_{ij}^{2}=1\quad\text{for all }i>j,\\
			\HH\rtimes\langle\sigma\rangle,     &\lambda=-1\ \text{and }p_{ij}=p_{\bar{i},\bar{j}} \quad\text{for all }i>j,\\
			\HH,  &\text{otherwise}.
		\end{cases}
	\]
	The first two cases are mutually exclusive. 
\end{theorem}

Our proof builds on the normal-element filtration used by Gaddis \cite{Gaddis2017} to solve the isomorphism problem; see Section~\ref{sec:preliminaries}.
We first classify the graded automorphisms when $n=2$ in Section~\ref{sec:two}.
For $n\ge3$, we divide the generators into corner, edge, and interior generators, and use the $2\times 2$ classification to determine the projections of the four corner images in Section~\ref{sec:corners}. 
In Section~\ref{sec:classification}, we determine the generator images from corners to edges and then to the interior, using relations involving opposite edges and the normal-element filtration.
The same arguments also give an isomorphism criterion for multiparameter quantum matrix algebras with deformation parameter $-1$; see Corollary~\ref{cor:minus-one-isomorphisms}.

In Section~\ref{sec:full}, we combine Theorem~\ref{thm:main} with the rigidity theorem of Goodearl and Yakimov \cite[Theorem~4.3]{GY2015} to determine the full automorphism group under a torsion-free hypothesis. More precisely, if the subgroup of $\kk^*$ generated by $\lambda$ and the $p_{ij}$ is torsion-free, then
\[
    \Aut(A)=\Autgr(A)=
    \begin{cases}
        \HH\rtimes\langle\tau\rangle,     &\lambda p_{ij}^{2}=1 \quad\text{for all }i>j,\\
        \HH,                              &\text{otherwise};
    \end{cases}
\]
see Corollary~\ref{cor:full}.
Specializations and examples are given in Section~\ref{sec:specializations}.

Throughout, all algebras and maps are over the base field $\kk$. 
For the standard grading $A=\bigoplus_{i\in\mathbb N_0}A_i$, we write $A_{\ge d}=\bigoplus_{i\ge d}A_i$ for $d\in\mathbb N_0$.

\section{Preliminaries}\label{sec:preliminaries}

\subsection{Standard basis and torus automorphisms}\label{subsec:torus}
By \cite[Theorem~2, p.~881]{AST1991}, $A$ is a domain and the ordered monomials
\[
	x_{11}^{r_{11}}x_{12}^{r_{12}}\cdots x_{nn}^{r_{nn}},	\quad r_{ij}\in\mathbb N_{0},
\]
form a $\kk$-basis of $A$, which we call the standard basis.

An automorphism of $A$ is called diagonal if it sends each standard generator to a nonzero scalar multiple of itself. 
\begin{lemma}\label{lem:diagonal}
	Suppose $\lambda\ne1$. The diagonal automorphisms of $A$ are precisely the torus automorphisms.
\end{lemma}
\begin{proof}
	Every torus automorphism is diagonal by definition. Conversely, let $\phi$ be a diagonal automorphism of $A$, say $\phi(x_{ij})=c_{ij}x_{ij}$ with $c_{ij}\in\kk^*$. For $1<i,j\le n$, the defining relations \eqref{eq:relations} give
	\[
		x_{ij}x_{11}=p_{i1}p_{1j}x_{11}x_{ij}+(\lambda-1)p_{i1}x_{1j}x_{i1}.
	\]
	Applying $\phi$ and comparing the resulting equality with the relation above, we obtain
	\[
		(\lambda-1)p_{i1}(c_{11}c_{ij}-c_{1j}c_{i1})x_{1j}x_{i1}=0.
	\]
	Since $(\lambda-1)p_{i1}\ne0$ and $x_{1j}x_{i1}$ is a nonzero basis monomial, it follows that $c_{11}c_{ij}=c_{1j}c_{i1}$. This equality also holds when $i=1$ or $j=1$. Set
	\[
		\alpha_i=\frac{c_{i1}}{c_{11}},\quad \beta_j=c_{1j},\quad 1\le i,j\le n.
	\]
	Then $c_{ij}=\alpha_i\beta_j$ for all $i,j$, so $\phi=h_{\alpha,\beta}\in\HH$.
\end{proof}

\subsection{The normal-element filtration}\label{subsec:normal}

Following Gaddis \cite{Gaddis2017}, we set $I_{-1}=0$ and define ideals $I_0 \subset I_1 \subset \cdots \subset I_{n-1}$ of $A$ such that $I_k/I_{k-1}$ is the ideal generated by all homogeneous degree-one normal elements in $A/I_{k-1}$.
The coefficient comparisons in the proof of \cite[Theorem~3.4]{Gaddis2017} require only $\lambda\ne1$. Thus the following description also holds when $\lambda=-1$.

\begin{theorem}[{\cite[Theorem~3.4]{Gaddis2017}}]\label{thm:normal-filtration}
    For $\lambda\ne1$ and $0\le k<n$,
    \[
        I_k/I_{k-1}=\langle x_{ij}:|i-j|=n-k-1\rangle \quad\text{in }A/I_{k-1}.
    \]
\end{theorem}

According to this theorem, one has
\begin{equation}\label{eq:filtration}
      I_k=\langle x_{ij}:|i-j|\ge n-k-1\rangle, \quad 0\le k< n.
\end{equation}
In particular,
\[
      I_0=\langle x_{1n},x_{n1}\rangle, \quad I_{n-2}=\langle x_{ij}:i\ne j\rangle,
\]
and the degree-one component of $I_k$ is
\begin{equation}\label{eq:filtration-degree-one}
      (I_k)_1=\operatorname{span}_{\kk}\{x_{ij}:|i-j|\ge n-k-1\}.
\end{equation}
This filtration gives the following restrictions on graded isomorphisms.

Let $B = \Mat{\mu}{\mathbf q}$ with $\mu\ne1$ and standard generators $\{y_{ij}\}$. 
Write $I_k(A)$ and $I_k(B)$ for the corresponding ideals in $A$ and $B$ defined above, respectively.
\begin{lemma}\label{lem:filtration}
      Let $\Phi:A\longrightarrow B$ be a graded isomorphism. Then 

      {\rm (1)} $\Phi(I_k(A))=I_k(B)$ for all $0\le k< n$.
      
      {\rm (2)} If $\Phi(x_{ij})=\sum\limits_{s,t}a^{ij}_{st}y_{st}$, then
      \begin{equation}\label{eq:support-bound}
            a^{ij}_{st}\ne0\quad\Longrightarrow\quad |s-t|\ge|i-j|.
      \end{equation}

      {\rm (3)} The two images $\Phi(x_{1n})$ and $\Phi(x_{n1})$ are nonzero scalar multiples of $y_{1n}$ and $y_{n1}$, in one order or the other.
\end{lemma}
\begin{proof}
      (1) We use induction on $k$. 
      A graded isomorphism preserves degree-one normal elements, hence $\Phi(I_0(A))=I_0(B)$. 
      Suppose $\Phi(I_{k-1}(A))=I_{k-1}(B)$ for some $1\le k<n$. Then $\Phi$ induces a graded isomorphism $A/I_{k-1}(A)\to B/I_{k-1}(B)$.
      This isomorphism preserves degree-one normal elements, so it maps $I_k(A)/I_{k-1}(A)$ onto $I_k(B)/I_{k-1}(B)$. Hence $\Phi(I_k(A))=I_k(B)$. 
 
      (2) Put $k=n-1-|i-j|$, so that $0\le k<n$. By \eqref{eq:filtration}, $x_{ij}\in (I_k(A))_1$. Part (1) and \eqref{eq:filtration-degree-one} imply
      \[
            \Phi(x_{ij})\in (I_k(B))_1=\operatorname{span}_{\kk}\{y_{st}:|s-t|\ge |i-j|\}.
      \]
      Since the generators $\{y_{st}\}$ are linearly independent, one has $a^{ij}_{st}=0$ whenever $|s-t|<|i-j|$. This proves \eqref{eq:support-bound}.

      (3) This follows from (1) and the proof of \cite[Lemma~5.3(3)]{Gaddis2017}, which uses only $\lambda\ne1$.
\end{proof}

\subsection{Standard isomorphisms} 
We use two graded isomorphisms; see \cite[Propositions~5.1--5.2]{Gaddis2017}. If $B=\Mat{\lambda}{\mathbf q}$ with ${\mathbf q}$ satisfying
\[
      q_{ij}=\lambda^{-1}p_{ji},\quad i>j,
\]
then
\begin{equation}\label{eq:transpose-iso}
      \Phi_{\tau}:A\longrightarrow B, \quad x_{ij}\longmapsto y_{ji},
\end{equation}
is a graded isomorphism. Similarly, if $B=\Mat{\lambda^{-1}}{\mathbf q}$ and
\[
      q_{ij}=p_{\bar{i},\bar{j}},\quad i>j,
\]
then
\begin{equation}\label{eq:reversal-iso}
      \Phi_{\sigma}:A\longrightarrow B, \quad x_{ij}\longmapsto y_{\bar{i},\bar{j}},
\end{equation}
is a graded isomorphism. These assertions can also be checked directly from \eqref{eq:relations}.


Recall the assignments on the standard generators of $A$:
\[
      \tau(x_{ij})=x_{ji},\quad \sigma(x_{ij})=x_{\bar{i},\bar{j}}.
\]
We now determine when $\tau$ and $\sigma$ define automorphisms of $A$.

\begin{proposition}\label{prop:symmetries}
      {\rm(1)} The assignment $\tau$ extends to an automorphism of $A$ if and only if $\lambda p_{ij}^2=1$ for all $i>j$. 

      {\rm(2)} The assignment $\sigma$ extends to an automorphism of $A$ if and only if $\lambda=-1$ and $p_{ij}=p_{\bar{i},\bar{j}}$ for all $i>j$. 

      {\rm(3)} The two assignments cannot both extend to automorphisms of $A$.
\end{proposition}
\begin{proof} We fix $n\ge i>j\ge 1$ throughout the proof. Thus by definition one has $\bar{i}<\bar{j}$.

      (1) Suppose $\tau$ is an automorphism. Applying it to the first-row relation $x_{1i}x_{1j}=p_{ji}x_{1j}x_{1i}$ gives $x_{i1}x_{j1}=p_{ji}x_{j1}x_{i1}$. The first-column relation is $x_{i1}x_{j1}=\lambda p_{ij}x_{j1}x_{i1}$. Combining these identities gives $\lambda p_{ij}^2=1$.
      Conversely, if $\lambda p_{ij}^2=1$ for all $i>j$, then $\tau$ is an automorphism of $A$ by \eqref{eq:transpose-iso}.

      (2) Suppose $\sigma$ is an automorphism. Applying it to the first-row relation gives
      \[
            x_{n,\bar{i}}x_{n,\bar{j}}=p_{ji}x_{n,\bar{j}}x_{n,\bar{i}}.
      \]
      The last-row relation is $x_{n,\bar{j}}x_{n,\bar{i}} =p_{\bar{i},\bar{j}}x_{n,\bar{i}}x_{n,\bar{j}}$.
      Combining these identities gives $p_{ij}=p_{\bar{i},\bar{j}}$.
      Applying $\sigma$ to the first-column relation gives
      \[
            x_{\bar{i},n}x_{\bar{j},n}  =\lambda p_{ij}x_{\bar{j},n}x_{\bar{i},n}.
      \]
      The last-column relation is  $x_{\bar{j},n}x_{\bar{i},n} =\lambda p_{\bar{j},\bar{i}}x_{\bar{i},n}x_{\bar{j},n}$.
      Combining these identities gives $1=\lambda^2p_{ij}p_{\bar{j},\bar{i}}=\lambda^2$.
      Since $\lambda\ne1$, we get $\lambda=-1$. 

      Conversely, if $\lambda=-1$ and $p_{ij}=p_{\bar{i},\bar{j}}$  for all $i>j$, then $\sigma$ is an automorphism of $A$ by \eqref{eq:reversal-iso}.

      (3) If both maps were automorphisms, part (2) would give $p_{n1}=p_{1n}$, hence $p_{n1}^2=1$, and part (1) would give $\lambda=1$, a contradiction.
\end{proof}

\section{The two-by-two case}\label{sec:two} 
In this section, let $A=\mathcal O_{\lambda,\mathbf p}(M_2(\kk))$ and write $p=p_{21}$. The notation $\HH$, $\tau$, and $\sigma$ refers to the group of torus automorphisms, transposition, and reversal, respectively, in the case $n=2$. 
The defining relations of $A$ are
\begin{equation}\label{eq:two-relations}
      \begin{gathered}
      x_{12}x_{11}=p^{-1}x_{11}x_{12},\quad x_{21}x_{11}=\lambda p x_{11}x_{21},\\
      x_{22}x_{12}=\lambda p x_{12}x_{22},\quad x_{22}x_{21}=p^{-1}x_{21}x_{22},\\
      x_{21}x_{12}=\lambda p^2x_{12}x_{21},\quad x_{22}x_{11}-x_{11}x_{22}=(\lambda-1)p x_{12}x_{21}.
      \end{gathered}
\end{equation}
By Lemma~\ref{lem:filtration}(3), any graded automorphism $\phi$ of $A$ either preserves or swaps the one-dimensional subspaces $\kk x_{12}$ and $\kk x_{21}$. 
The following lemma treats the first case.

\begin{lemma}\label{lem:two-fixed}
      If $\phi(x_{12})=\beta x_{12}$ and $\phi(x_{21})=\gamma x_{21}$ for a graded automorphism $\phi$, then $\phi(x_{11})=\alpha x_{11}$ and $\phi(x_{22})=\delta x_{22}$ for some $\alpha,\delta\in\kk^*$.
\end{lemma}

\begin{proof}
      Write
      \begin{align*}
            \phi(x_{11})&=a_{11}x_{11}+a_{12}x_{12}+a_{21}x_{21}+a_{22}x_{22},\\
            \phi(x_{22})&=d_{11}x_{11}+d_{12}x_{12}+d_{21}x_{21}+d_{22}x_{22}.
      \end{align*}
      Applying $\phi$ to the relations of $x_{12}$ with $x_{11}$ and $x_{22}$ gives
      \[
            x_{12}\phi(x_{11})=p^{-1}\phi(x_{11})x_{12},\quad \phi(x_{22})x_{12}=\lambda p x_{12}\phi(x_{22}).
      \]
      Comparing the coefficients of $x_{12}x_{22}$ in the first equation and $x_{11}x_{12}$ in the second yields
      \[
            (1-\lambda)a_{22}=(1-\lambda)d_{11}=0.
      \]
      Thus $a_{22}=d_{11}=0$. 
      The induced graded automorphism of $A/\langle x_{12},x_{21}\rangle$ sends $x_{11}$ to $a_{11}x_{11}$ and $x_{22}$ to $d_{22}x_{22}$. Thus $a_{11}d_{22}\ne0$.

      Applying $\phi$ to the first two relations in \eqref{eq:two-relations} gives
      \[
          x_{12}\phi(x_{11}) =p^{-1}\phi(x_{11}) x_{12},\quad x_{21}\phi(x_{11}) = \lambda p \phi(x_{11})x_{21}.
      \]
      Comparing the coefficients of $x_{12}^2$ and $x_{21}^2$, respectively, we obtain
      \begin{equation}\label{eq:two-1}
            a_{12}(p^{-1}-1)=0,\quad a_{21}(\lambda p-1)=0.
      \end{equation}
      Since $d_{22}\ne0$, comparing the coefficients of $x_{12}x_{22}$ and $x_{21}x_{22}$ on both sides of
      \[
            \phi(x_{22})\phi(x_{11})-\phi(x_{11})\phi(x_{22}) =(\lambda-1)p\beta\gamma x_{12}x_{21},
      \]
      we obtain 
      \begin{equation}\label{eq:two-2}
           a_{12}(\lambda p-1)=0,\quad a_{21}(p^{-1}-1)=0.
      \end{equation}
      Since $\lambda\ne1$, the factors $p^{-1}-1$ and $\lambda p-1$ cannot both be zero. Thus \eqref{eq:two-1} and \eqref{eq:two-2} give $a_{12}=a_{21}=0$.

      Now $\phi(x_{11})=a_{11}x_{11}$ and $\phi(x_{22})=d_{12}x_{12}+d_{21}x_{21}+d_{22}x_{22}$. Applying $\phi$ to the last relation in \eqref{eq:two-relations} gives
      \[
            a_{11}(d_{12}(p^{-1}-1)x_{11}x_{12} +d_{21}(\lambda p-1)x_{11}x_{21} +d_{22}(\lambda-1)p x_{12}x_{21}) =(\lambda-1)p\beta\gamma x_{12}x_{21}.
      \]
      Since $a_{11}\ne0$, comparing the coefficients of $x_{11}x_{12}$ and $x_{11}x_{21}$ gives
      \begin{equation}\label{eq:two-3}
            d_{12}(p^{-1}-1)=0,\quad d_{21}(\lambda p-1)=0.
      \end{equation}
      Applying $\phi$ to the relations of $x_{22}$ with $x_{12}$ and $x_{21}$ gives
      \[
            \phi(x_{22})x_{12}=\lambda p x_{12}\phi(x_{22}),\quad \phi(x_{22})x_{21}=p^{-1}x_{21}\phi(x_{22}).
      \]
      Comparing the coefficients of $x_{12}^2$ in the first equation and $x_{21}^2$ in the second gives
      \begin{equation}\label{eq:two-4}
            d_{12}(\lambda p-1)=0,\quad d_{21}(p^{-1}-1)=0.
      \end{equation}
      Since $\lambda\ne1$, the factors $p^{-1}-1$ and $\lambda p-1$ cannot both be zero. Thus \eqref{eq:two-3} and \eqref{eq:two-4} give $d_{12}=d_{21}=0$.
      Taking $\alpha=a_{11}$ and $\delta=d_{22}$ proves the assertion.
\end{proof}

We now determine $\Autgr(A)$ by considering automorphisms that swap $\kk x_{12}$ and $\kk x_{21}$.
\begin{proposition}\label{prop:two}
      The graded automorphism group of $A$ is
      \[
            \Autgr(A)=
            \begin{cases}
                  \HH\rtimes\langle\tau\rangle,&\lambda p^2=1,\\
                  \HH\rtimes\langle\sigma\rangle,&\lambda=-1,\ p^2=1,\\
                  \HH,&\text{otherwise}.
            \end{cases}
      \]
\end{proposition}

\begin{proof}
      Let $\phi\in\Autgr(A)$. By Lemma~\ref{lem:filtration}(3), $\phi$ either preserves or swaps the subspaces $\kk x_{12}$ and $\kk x_{21}$.

      {\bf Case 1: $\phi$ preserves $\kk x_{12}$ and $\kk x_{21}$.}
      Lemma~\ref{lem:two-fixed} shows that $\phi$ is diagonal on all four generators. Hence $\phi\in\HH$ by Lemma~\ref{lem:diagonal}.

      {\bf Case 2: $\phi$ swaps $\kk x_{12}$ and $\kk x_{21}$.}
      Write $\phi(x_{12})=\beta x_{21}$ and $\phi(x_{21})=\gamma x_{12}$, where $\beta,\gamma\ne0$. Applying $\phi$ to $x_{21}x_{12}=\lambda p^2x_{12}x_{21}$ gives
      \[
            \beta\gamma x_{12}x_{21} =\lambda p^2\beta\gamma x_{21}x_{12}=\lambda^2p^4\beta\gamma x_{12}x_{21}.
      \]
      Put $t=\lambda p^2$, so $t^2=1$. 
      If $t=1$, Proposition~\ref{prop:symmetries} shows that $\tau$ is an automorphism. Since $\phi\circ\tau$ preserves $\kk x_{12}$ and $\kk x_{21}$, Case 1 gives $\phi\circ\tau\in\HH$, and hence $\phi\in\HH\tau$.

      Now suppose $t\ne1$. Then $t=-1$. Write $\phi(x_{11})=a_{11}x_{11}+a_{12}x_{12}+a_{21}x_{21}+a_{22}x_{22}$. Applying $\phi$ to $x_{12}x_{11}=p^{-1}x_{11}x_{12}$ and canceling $\beta$ gives
      \[
            x_{21}\phi(x_{11})=p^{-1}\phi(x_{11})x_{21}.
      \]
      Comparing the coefficient of $x_{11}x_{21}$ gives $a_{11}(\lambda p-p^{-1})=0$. Since $\lambda p-p^{-1}=p^{-1}(t-1)\ne0$, we obtain $a_{11}=0$. The induced graded automorphism of $A/\langle x_{12},x_{21}\rangle$ sends $x_{11}$ to $a_{22}x_{22}$. Thus $a_{22}\ne0$.
      Comparing the coefficient of $x_{21}x_{22}$ in the same equation gives
      \[
            a_{22}=p^{-2}a_{22}.
      \]
      Hence $p^2=1$ and $\lambda=-1$. By Proposition~\ref{prop:symmetries}, $\sigma$ is an automorphism in this case. Since $\phi\circ\sigma$ preserves $\kk x_{12}$ and $\kk x_{21}$, Case 1 gives $\phi\circ\sigma\in\HH$, and hence $\phi\in\HH\sigma$.

      Conversely, the indicated compositions are automorphisms whenever their parameter conditions hold. These conditions are disjoint by Proposition~\ref{prop:symmetries}. Each involution normalizes $\HH$ and generates a subgroup intersecting $\HH$ trivially, giving the asserted semidirect products.
\end{proof}

\section{Projection to the corner subalgebra}\label{sec:corners}
We distinguish the generators $\{x_{ij}\}$ of $A=\Mat{\lambda}{\mathbf p}$ according to their positions in the matrix $(x_{ij})$. 
We write the generators in matrix form for clarity:
\[  
    \begin{pmatrix}     
       x_{11} & \cdots  & x_{1j} & \cdots &x_{1n} \\
       \vdots &        & {\vdots} &        & \vdots \\
       x_{i1} & {\cdots} & {x_{ij}} & {\cdots} &{x_{in}} \\  
       {\vdots} &        & {\vdots} &        & {\vdots} \\
        x_{n1} & {\cdots} & {x_{nj}} & {\cdots} & {x_{nn}}
    \end{pmatrix}.
\]
We call $x_{11},x_{1n},x_{n1},x_{nn}$ the corner generators. The generators $x_{1j},x_{nj}$ with $1<j<n$, and $x_{i1},x_{in}$ with $1<i<n$, are called the edge generators. The remaining generators $x_{ij}$ with $1<i,j<n$ are called the interior generators.
Let
\[
      \Ccorner=\kk\langle x_{11},x_{1n},x_{n1},x_{nn}\rangle_{\rm alg}
\]
be the subalgebra of $A$ generated by the four corner generators. We call $\Ccorner$ the corner subalgebra. The subalgebra $\Ccorner$ inherits the grading of $A$, and its degree-one component is
\[
      \Ccorner_1 =\operatorname{span}_{\kk}\{x_{11},x_{1n},x_{n1},x_{nn}\}.
\]
The standard basis identifies $\Ccorner$ with the $2\times 2$ quantum matrix algebra with parameters $\lambda$ and $p_{n1}$.
Let $J$ be the ideal generated by all noncorner generators. Fixing the four corner generators and sending all other generators to zero gives an algebra homomorphism
\[
      \pi:A\longrightarrow\Ccorner.
\]
The map $\pi$ is the identity on $\Ccorner$ and induces an isomorphism $A/J\cong\Ccorner$.

In the proof of Theorem~\ref{thm:main}, we use standard isomorphisms to adjust the corner images.
These isomorphisms may change the parameters, so we work with graded isomorphisms in this section and the next.

Let $B=\Mat{\mu}{\mathbf q}$ with $\mu\ne1$ and standard generators $\{y_{ij}\}$.
Denote its corner subalgebra and corner projection by $\Ccorner_B$ and $\pi_B$, respectively.

\begin{lemma}\label{lem:corner-projection}
      For every graded isomorphism $\Phi:A\to B$, the map
      \[
            \psi=\pi_B\circ\Phi|_{\Ccorner}:\Ccorner\longrightarrow\Ccorner_B, \quad c\mapsto\pi_B(\Phi(c))
      \] 
      is a graded isomorphism.
\end{lemma}
\begin{proof}
      The standard basis shows that $\Ccorner$ and $\Ccorner_B$ have the same finite dimension in each degree. Since both algebras are generated in degree one, it suffices to show that $\psi$ is injective on $\Ccorner_1$.
      Let $u\in\Ccorner_1$ satisfy $\psi(u)=0$. Since $\pi_B(\Phi(u))=0$, $\Phi(u)$ is a linear combination of noncorner generators of $B$. We first claim that $\Phi(u)$ is a linear combination of off-diagonal generators.

      By Lemma~\ref{lem:filtration}(3), $\Phi(x_{1n}),\Phi(x_{n1})\in\kk y_{1n}\oplus\kk y_{n1}$.
      Since $u\in\Ccorner_1$, it suffices to show that the coefficients of $y_{ii}$ in $\Phi(x_{11})$ and $\Phi(x_{nn})$ are zero for $1<i<n$. Set $z_A=x_{1n}x_{n1}$ and $z_B=y_{1n}y_{n1}$. By Lemma~\ref{lem:filtration}(3) again, $\Phi(z_A)$ is a nonzero scalar multiple of either $y_{1n}y_{n1}$ or $y_{n1}y_{1n}$.
      Since
      \[
            y_{n1}y_{1n}=\mu q_{n1}^2y_{1n}y_{n1},
      \]
      one has $\Phi(z_A)\in\kk^*z_B$.
      The defining relations yield
      \[
            x_{11}z_A=\lambda^{-1}z_Ax_{11},\quad x_{nn}z_A=\lambda z_Ax_{nn},\quad y_{ii}z_B=z_By_{ii},\quad 1<i<n.
      \]
      Write $\Phi(x_{11})=\sum_{s,t}c_{st}y_{st}$. Applying $\Phi$ to the first relation and canceling the nonzero scalar in $\Phi(z_A)$ gives
      \[
            \Phi(x_{11})z_B=\lambda^{-1}z_B\Phi(x_{11}).
      \]
      Each generator of $B$ commutes up to a nonzero scalar with $z_B$, and the products $z_By_{st}$ are linearly independent. Comparing the coefficients of $z_By_{ii}$ gives $c_{ii}=\lambda^{-1}c_{ii}$ for $1<i<n$. Since $\lambda\ne1$, these coefficients vanish. The same argument applied to $x_{nn}z_A=\lambda z_Ax_{nn}$ excludes interior diagonal terms from $\Phi(x_{nn})$, proving the claim.

      Thus $\Phi(u)\in I_{n-2}(B)$. By Lemma~\ref{lem:filtration}(1), $\Phi(I_{n-2}(A))=I_{n-2}(B)$, so
      \[
            u\in\Ccorner_1\cap I_{n-2}(A)=\kk x_{1n}\oplus\kk x_{n1}.
      \]
      Hence $\Phi(u)\in\kk y_{1n}\oplus\kk y_{n1}$, and $\pi_B$ is the identity on this subspace. Thus $\Phi(u)=\psi(u)=0$, so $u=0$. This proves the lemma.
\end{proof}

In particular, for $\phi\in\Autgr(A)$, the map $\pi\circ\phi|_{\Ccorner}$ is a graded automorphism of $\Ccorner$.
By Proposition~\ref{prop:two}, it is a torus automorphism, possibly composed with transposition or reversal.

\section{The graded classification}\label{sec:classification}

We retain the notation of Section~\ref{sec:corners}.
We compose $\phi\in\Autgr(A)$ with a suitable standard isomorphism so that the resulting map has diagonal corner projection.
The next two lemmas show that this map is diagonal, first on the four corner generators and then on all generators.

\begin{lemma}[Lifting a diagonal corner action]\label{lem:corner-lift}
      Let $\Phi:A\to B$ be a graded isomorphism. Suppose
      \[
            \Phi(x_{11})=\alpha y_{11}+U, \quad\Phi(x_{1n})=\beta y_{1n},\quad \Phi(x_{n1})=\gamma y_{n1},\quad \Phi(x_{nn})=\delta y_{nn}+V,
      \]
      where $\alpha,\beta,\gamma,\delta\ne0$, and $U,V$ are linear combinations of noncorner generators. Then $U=V=0$.
\end{lemma}
\begin{proof}
      Let $u_{ij}$ and $v_{ij}$ be the coefficients of $y_{ij}$ in $U$ and $V$, respectively. We prove that $u_{ij}=v_{ij}=0$ for all $i,j$ by comparing coefficients in the standard basis of $B$.

      Applying $\Phi$ to the relations of $x_{11}$ and $x_{nn}$ with $x_{1n}$ and $x_{n1}$, and canceling $\beta$ and $\gamma$, gives
      \begin{equation}\label{eq:lift-normal}
            \begin{aligned}
                  \Phi(x_{11})y_{1n}&=p_{n1}y_{1n}\Phi(x_{11}),   &\Phi(x_{11})y_{n1}=\lambda^{-1}p_{1n}y_{n1}\Phi(x_{11}),\\
                  \Phi(x_{nn})y_{1n}&=\lambda p_{n1}y_{1n}\Phi(x_{nn}),  &\Phi(x_{nn})y_{n1}=p_{1n}y_{n1}\Phi(x_{nn}).
            \end{aligned}
      \end{equation}
      Since $\alpha\ne0$, comparing the coefficients of $y_{11}y_{1n}$ in the first equation and $y_{11}y_{n1}$ in the second gives
      \[
            q_{1n}=p_{1n},\quad \mu q_{n1}=\lambda p_{n1}.
      \]
      Thus $q_{n1}=p_{n1}$ and $\mu=\lambda$.
      Applying $\Phi$ to the relation between $x_{nn}$ and $x_{11}$ gives
      \begin{equation}\label{eq:lift-cross}
            \Phi(x_{nn})\Phi(x_{11})-\Phi(x_{11})\Phi(x_{nn}) =(\lambda-1)p_{n1}\beta\gamma y_{1n}y_{n1}.
      \end{equation}

      \textbf{Step 1: Eliminating the adjacent edge terms.}
      Fix $1<j<n$. We compare the coefficient of $y_{1n}y_{nj}$ in \eqref{eq:lift-cross}.
      Since neither $\Phi(x_{11})$ nor $\Phi(x_{nn})$ contains $y_{1n}$, the defining relations show that $y_{1n}y_{nj}$ can only arise from rewriting the product $y_{nn}y_{1j}$ in standard order.

      Since $y_{nn}$ occurs only in $\Phi(x_{nn})$, with coefficient $\delta$, the product $y_{nn}y_{1j}$ occurs in $\Phi(x_{nn})\Phi(x_{11})$ with coefficient $\delta u_{1j}$. 
      Thus the coefficient of $y_{1n}y_{nj}$ in $\Phi(x_{nn})\Phi(x_{11})$ is $\delta u_{1j}(\lambda-1)p_{n1}$.
      In $\Phi(x_{11})\Phi(x_{nn})$, these two generators occur in the order $y_{1j}y_{nn}$, which is already in standard order. Hence the coefficient of $y_{1n}y_{nj}$ in $\Phi(x_{11})\Phi(x_{nn})$ is zero.
      Since the coefficient of $y_{1n}y_{nj}$ on the right side of \eqref{eq:lift-cross} is zero, we obtain
      \[
            \delta u_{1j}(\lambda-1)p_{n1}=0.
      \]
      Hence $u_{1j}=0$.
      
      Similarly, comparing the coefficients of $y_{in}y_{n1}$, $y_{1j}y_{n1}$ and $y_{1n}y_{i1}$ in \eqref{eq:lift-cross} gives, respectively,
      \[
            \delta u_{i1}(\lambda-1)q_{ni}=0,\quad  \alpha v_{nj}(\lambda-1)p_{n1}=0,\quad \alpha v_{in}(\lambda-1)q_{i1}=0,
      \]
      where $1<i,j<n$. Hence $u_{i1}=v_{nj}=v_{in}=0$.

      \textbf{Step 2: Eliminating the remaining edge terms.}
      Fix $1<i,j<n$. Each of the monomials
      \[
            y_{nj}y_{nn},\quad y_{in}y_{nn},\quad y_{11}y_{1j},\quad y_{11}y_{i1}
      \]
      is a product of two generators from the same row or column.
      Hence no other pair of generators can contribute to its coefficient in \eqref{eq:lift-cross}.
      Comparing the coefficients of these four terms in \eqref{eq:lift-cross} and canceling the nonzero scalars $\delta$ and $\alpha$ gives
      \begin{equation}\label{eq:lift-edge-equations}
            \begin{aligned}
                  u_{nj}(q_{jn}-1)&=0, &u_{in}(\lambda q_{ni}-1)&=0,\\
                  v_{1j}(q_{1j}-1)&=0, &v_{i1}(\lambda q_{i1}-1)&=0.
            \end{aligned}
      \end{equation}

      Note that both $y_{1n}$ and $y_{n1}$ commute with every generator up to a nonzero scalar.
      Comparing the coefficients of $y_{1n}y_{nj}$ and $y_{in}y_{n1}$ in the first two equations of \eqref{eq:lift-normal}, respectively, gives
      \[
            u_{nj}(\lambda q_{nj}-1)=0,\quad u_{in}(1-q_{ni})=0.
      \]
      Comparing the coefficients of $y_{1n}y_{i1}$ and $y_{1j}y_{n1}$ in the last two equations of \eqref{eq:lift-normal}, respectively, gives
      \[
            v_{i1}(q_{i1}-1)=0,\quad v_{1j}(1-\lambda q_{j1})=0.
      \]
      If $u_{nj}\ne0$, the equations above and \eqref{eq:lift-edge-equations} give $q_{jn}=1$ and $\lambda q_{nj}=1$.
      Since $q_{jn}q_{nj}=1$, this implies $\lambda=1$, a contradiction. Hence $u_{nj}=0$.
      The same argument applied to the other three coefficients gives $u_{in}=v_{1j}=v_{i1}=0$.

      \textbf{Step 3: Eliminating interior terms.}
      Now $U$ and $V$ are linear combinations of the interior generators.
      Fix $1<i,j<n$. As in Step 1, the monomial $y_{in}y_{nj}$ in \eqref{eq:lift-cross} can only arise from rewriting $y_{nn}y_{ij}$, which occurs in $\Phi(x_{nn})\Phi(x_{11})$ with coefficient $\delta u_{ij}$.
      Comparing the coefficient of $y_{in}y_{nj}$ in \eqref{eq:lift-cross} gives
      \[
            \delta u_{ij}(\lambda-1)q_{ni}=0.
      \]
      Thus $u_{ij}=0$ for all $1<i,j<n$, so $U=0$.

      Now $\Phi(x_{11})=\alpha y_{11}$.
      Similarly, comparing the coefficient of $y_{1j}y_{i1}$ in \eqref{eq:lift-cross}, which comes from rewriting $y_{ij}y_{11}$, gives
      \[
            \alpha v_{ij}(\lambda-1)q_{i1}=0.
      \]
      Hence $V=0$. This proves the lemma.
\end{proof}

The next lemma extends the conclusion of Lemma~\ref{lem:corner-lift} from the corner generators to all standard generators. It also shows that the source and target parameters coincide.
\begin{lemma}\label{lem:corner-rigidity}
      Let $\Phi:A\to B$ be a graded isomorphism. Suppose
      \[
            \Phi(x_{11})=\alpha y_{11},\quad\Phi(x_{1n})=\beta y_{1n},\quad\Phi(x_{n1})=\gamma y_{n1},\quad\Phi(x_{nn})=\delta y_{nn},
      \]
      where $\alpha,\beta,\gamma,\delta\ne0$. Then
      \[
            \Phi(x_{ij})=c_{ij}y_{ij}, \quad 1\le i,j\le n
      \]
      for nonzero scalars $c_{ij}$. Moreover, $\mu=\lambda$, $\mathbf p=\mathbf q$, and there are $\alpha_i,\beta_j\in\kk^*$ such that
      $c_{ij}=\alpha_i\beta_j$.
\end{lemma}
\begin{proof}
      As in Lemma~\ref{lem:corner-lift}, comparing the corner relations gives
      $\mu=\lambda$ and $q_{n1}=p_{n1}$.

      \textbf{Step 1: Top and bottom edges.}
      Fix $1<j<n$ and write
      \[
            \Phi(x_{1j})=\sum_{s,t}a_{st}y_{st},\quad \Phi(x_{nj})=\sum_{s,t}b_{st}y_{st}.
      \]
      Applying $\Phi$ to $x_{1j}x_{11}=p_{1j}x_{11}x_{1j}$ and canceling $\alpha$ gives
      \[
            \Phi(x_{1j})y_{11}=p_{1j}y_{11}\Phi(x_{1j}).
      \]
      For $1<s,t\le n$, the defining relations give
      \[
            y_{st}y_{11} =q_{s1}q_{1t}y_{11}y_{st} +(\lambda-1)q_{s1}y_{1t}y_{s1}.
      \]
      The monomial $y_{1t}y_{s1}$ can only arise from rewriting $a_{st}y_{st}y_{11}$.
      Comparing its coefficient gives
      \[
            a_{st}(\lambda-1)q_{s1}=0.
      \]
      Hence $a_{st}=0$ for $1<s,t\le n$.
      By Lemma~\ref{lem:filtration}(2), the coefficient of $y_{11}$ is also zero.
      Thus $\Phi(x_{1j})$ is a linear combination of generators in the first row and first column, other than $y_{11}$.
      Applying the same argument to
      \[
            y_{nn}\Phi(x_{nj})=p_{jn}\Phi(x_{nj})y_{nn}
      \]
      shows that $\Phi(x_{nj})$ is a linear combination of generators in the last row and last column, other than $y_{nn}$.

      Applying $\Phi$ to the relations between $x_{nn}$ and $x_{1j}$, and between $x_{nj}$ and $x_{11}$, gives
      \begin{align}
            \delta(y_{nn}\Phi(x_{1j}) -p_{n1}p_{jn}\Phi(x_{1j})y_{nn}) &=(\lambda-1)p_{n1}\beta y_{1n}\Phi(x_{nj}),
            \label{eq:edge-top}\\
            \alpha(\Phi(x_{nj})y_{11} -p_{n1}p_{1j}y_{11}\Phi(x_{nj})) &=(\lambda-1)p_{n1}\gamma\Phi(x_{1j})y_{n1}.
            \label{eq:edge-bottom}
      \end{align}
      Fix $1<s<n$. Every basis monomial on the right side of \eqref{eq:edge-top} contains $y_{1n}$, so the coefficient of $y_{sn}y_{n1}$ is zero.
      Similarly, the coefficient of $y_{1n}y_{s1}$ on the right side of \eqref{eq:edge-bottom} is zero.
      Comparing these coefficients gives
      \[
            \delta a_{s1}(\lambda-1)q_{ns}=0,\quad
            \alpha b_{sn}(\lambda-1)q_{s1}=0.
      \]
      Hence $a_{s1}=b_{sn}=0$. We now have
      \[
            \begin{aligned}
                  \Phi(x_{1j}) &=\sum_{t=2}^{n-1}a_{1t}y_{1t} +a_{1n}y_{1n}+a_{n1}y_{n1},\\
                  \Phi(x_{nj}) &=\sum_{t=2}^{n-1}b_{nt}y_{nt} +b_{1n}y_{1n}+b_{n1}y_{n1}.
            \end{aligned}
      \]

      Neither $y_{1n}^2$ nor $y_{1n}y_{n1}$ occurs on the left side of \eqref{eq:edge-top}.
      Comparing their coefficients gives $b_{1n}=b_{n1}=0$.
      Similarly, comparing the coefficients of $y_{1n}y_{n1}$ and $y_{n1}^2$ in \eqref{eq:edge-bottom} gives $a_{1n}=a_{n1}=0$.

      For $1<t<n$, comparing the coefficient of $y_{1n}y_{nt}$ in \eqref{eq:edge-top} gives
      \begin{equation}\label{eq:paired-support}
            \delta a_{1t}=\beta b_{nt}.
      \end{equation}
      If $a_{1t}\ne0$, then $b_{nt}\ne0$.
      Applying \eqref{eq:support-bound} in Lemma~\ref{lem:filtration}(2) to these two coefficients gives
      \[
            t-1\ge j-1,\quad n-t\ge n-j.
      \]
      Thus $t=j$.
      Equation~\eqref{eq:paired-support} now gives $a_{1t}=b_{nt}=0$ for all $t\ne j$. Therefore
      \[
            \Phi(x_{1j})\in\kk^* y_{1j},\quad  \Phi(x_{nj})\in \kk^* y_{nj}.
      \]

      \textbf{Step 2: Left and right edges.}
      Let $\Phi_{\tau,A}:A\to A_\tau$ and $\Phi_{\tau,B}:B\to B_\tau$ be the standard isomorphisms in \eqref{eq:transpose-iso}.
      The graded isomorphism
      \[
            \widetilde\Phi =\Phi_{\tau,B}\circ\Phi\circ\Phi_{\tau,A}^{-1}:A_\tau\longrightarrow B_\tau
      \]
      maps each corner generator to a nonzero scalar multiple of the corner generator in the same position.
      Both $A_\tau$ and $B_\tau$ have deformation parameter $\lambda\ne1$, so Step~1 applies to $\widetilde\Phi$.
      Since $\Phi= \Phi_{\tau,B}^{-1}\circ\widetilde\Phi\circ\Phi_{\tau,A}$, we obtain
      \[
            \Phi(x_{i1})\in\kk^* y_{i1},\quad  \Phi(x_{in})\in \kk^* y_{in}, \quad 1<i<n.
      \]

      \textbf{Step 3: Interior generators.}
      For the corner and edge generators, write $\Phi(x_{ij})=c_{ij}y_{ij}$ with $c_{ij}\ne0$. Applying $\Phi$ to
      \[
            x_{1l}x_{1j}=p_{jl}x_{1j}x_{1l}, \quad 1\le j<l\le n
      \]
      and comparing with the corresponding relations in $B$ gives $p_{jl}=q_{jl}$. By multiplicative antisymmetry, $\mathbf p=\mathbf q$.

      For $1<i,j<n$, applying $\Phi$ to
      \[
            x_{nj}x_{i1}-p_{ni}p_{1j}x_{i1}x_{nj} =(\lambda-1)p_{ni}x_{ij}x_{n1}
      \]
      now gives
      \[
            c_{nj}c_{i1}(\lambda-1)p_{ni}y_{ij}y_{n1}=(\lambda-1)p_{ni}c_{n1}\Phi(x_{ij})y_{n1}.
      \]
      Since $(\lambda-1)p_{ni}c_{n1}\ne0$, one gets
      \[
            \Phi(x_{ij})=\frac{c_{nj}c_{i1}}{c_{n1}}y_{ij}.
      \]
      Thus every standard generator is mapped to a nonzero scalar multiple of the generator in the same position.
      Identifying $B$ with $A$, Lemma~\ref{lem:diagonal} gives $c_{ij}=\alpha_i\beta_j$ for some $\alpha_i,\beta_j\in\kk^*$.
\end{proof}


Now we can complete the proof of Theorem~\ref{thm:main}.
\begin{proof}[\bf Proof of Theorem~\ref{thm:main}]
      The case $n=2$ is Proposition~\ref{prop:two}. Assume $n\ge3$ and let $\phi\in\Autgr(A)$. By Lemma~\ref{lem:corner-projection}, applied with $B=A$, the map $\psi=\pi\circ\phi|_{\Ccorner}$ is an automorphism of the corner subalgebra. Proposition~\ref{prop:two} shows that $\psi$ is a torus automorphism, possibly composed with transposition or reversal.

      Choose a graded isomorphism $S:A\to B$ as follows. If $\psi$ is diagonal, take $B=A$ and $S={\rm id}_A$. In the transpose case, take $S=\Phi_\tau$ from \eqref{eq:transpose-iso}; in the reversal case, take $S=\Phi_\sigma$ from \eqref{eq:reversal-iso}. Write $B=\Mat{\mu}{\mathbf q}$ with standard generators $\{y_{ij}\}$, and let $\pi_B$ be its corner projection. Put $\Phi=S\circ\phi$.
      Since $S$ maps corner generators to corner generators and noncorner generators to noncorner generators, we have
      \[
            \pi_B\circ\Phi|_{\Ccorner}=(S|_{\Ccorner})\circ\psi.
      \]
      By the choice of $S$, this map sends each corner generator to a nonzero scalar multiple of the target corner generator in the same position.

      Lemma~\ref{lem:filtration}(3) shows that $\Phi(x_{1n})$ and $\Phi(x_{n1})$ contain no noncorner terms. Hence the hypotheses of Lemma~\ref{lem:corner-lift} hold, and that lemma removes the noncorner terms from $\Phi(x_{11})$ and $\Phi(x_{nn})$. 
      We can now apply Lemma~\ref{lem:corner-rigidity}, which gives
      \[
            \mu=\lambda,\quad \mathbf q=\mathbf p,\quad\Phi(x_{ij})=\alpha_i\beta_jy_{ij}, \quad 1\le i,j\le n
      \]
      for some $\alpha_i,\beta_j\in\kk^*$. Identifying $B$ with $A$, we obtain $\Phi\in\HH$.

      If $S={\rm id}_A$, then $\phi=\Phi\in\HH$.

      Suppose $S=\Phi_\tau$. By \eqref{eq:transpose-iso},
      the target parameter matrix satisfies
      \[
            q_{ij}=\lambda^{-1}p_{ji},\quad i>j.
      \]
      Since $\mathbf q=\mathbf p$, we obtain $p_{ij}=\lambda^{-1}p_{ji}$ for all $i>j$.
      Using $p_{ij}p_{ji}=1$, this gives
      \[
            \lambda p_{ij}^2=1, \quad i>j.
      \]
      Thus $\tau$ is an automorphism of $A$ by Proposition~\ref{prop:symmetries}.
      After identifying $B$ with $A$, the map $S$ is precisely $\tau$.
      Since $\tau^{-1}=\tau$, we have
      \[
            \phi=S^{-1}\circ\Phi=\tau\circ\Phi\in\langle\HH,\tau\rangle.
      \]

      Suppose $S=\Phi_\sigma$. By \eqref{eq:reversal-iso}, the target parameters satisfy
      \[
            \mu=\lambda^{-1},\quad q_{ij}=p_{\bar{i},\bar{j}}, \quad i>j.
      \]
      Combining these equalities with $\mu=\lambda$ and $\mathbf q=\mathbf p$ gives
      \[
            \lambda^2=1,\quad p_{ij}=p_{\bar{i},\bar{j}}, \quad i>j.
      \]
      Since $\lambda\ne1$, we obtain $\lambda=-1$.
      These are exactly the conditions for $\sigma$ to be an automorphism of $A$ in Proposition~\ref{prop:symmetries}.
      After identifying $B$ with $A$, we therefore have $S=\sigma$. Since $\sigma^{-1}=\sigma$, it follows that
      \[
            \phi=S^{-1}\circ\Phi=\sigma\circ\Phi\in\langle\HH,\sigma\rangle.
      \]
      In particular, if neither of these parameter conditions holds, then $S={\rm id}_A$ and $\phi\in\HH$.

      Conversely, Proposition~\ref{prop:symmetries} shows that the indicated involutions are automorphisms under the stated parameter conditions, and that the two cases are disjoint. 
      Each indicated involution normalizes $\HH$. Since it is a nontrivial permutation of the generators, the subgroup it generates intersects $\HH$ trivially. This proves the semidirect-product decompositions.
\end{proof}

The preceding arguments also extend the isomorphism criterion in \cite[Theorem~5.8]{Gaddis2017} to the case in which both deformation parameters are $-1$.

\begin{corollary}\label{cor:minus-one-isomorphisms}
      Suppose $\operatorname{char}\kk\ne2$. Then
      \[
            \mathcal O_{-1,\mathbf p}(M_n(\kk)) \cong \mathcal O_{-1,\mathbf q}(M_m(\kk))
      \]
      if and only if $n=m$ and one of the following conditions holds:
      \begin{enumerate}
            \item $q_{ij}=p_{ij}$ for all $i>j$;
            \item $q_{ij}=-p_{ji}$ for all $i>j$;
            \item $q_{ij}=p_{\bar{i},\bar{j}}$ for all $i>j$;
            \item $q_{ij}=-p_{\bar{j},\bar{i}}$ for all $i>j$.
      \end{enumerate}
\end{corollary}
\begin{proof}
      The sufficiency follows from the standard isomorphisms and their compositions.
      Conversely, let
      \[
            A=\mathcal O_{-1,\mathbf p}(M_n(\kk)),\quad B=\mathcal O_{-1,\mathbf q}(M_m(\kk)),
      \]
      and suppose $A\cong B$. Comparing GK dimensions gives $n=m$. By \cite[Theorem~0.1]{BZ2017}, there is a graded isomorphism $\Phi:A\to B$.
      In the following, we may replace $A$ by an isomorphic quantum matrix algebra using the standard isomorphisms. We keep the same notation $A$, $\mathbf p$, and $x_{ij}$.

      By Lemma~\ref{lem:filtration}(3), after composing $\Phi$ on the right with standard transposition if necessary, we may assume
      \[
            \Phi(x_{1n})\in\kk^*y_{1n},\quad \Phi(x_{n1})\in\kk^*y_{n1}.
      \]
      By Lemma~\ref{lem:corner-projection}, the corner projection is an isomorphism, so $y_{11}$ has a nonzero coefficient in $\Phi(x_{11})$ or $\Phi(x_{nn})$.       
      Composing $\Phi$ on the right with standard transposition and reversal if necessary, we may assume that the coefficient of $y_{11}$ in $\Phi(x_{11})$ is nonzero. We show that $\mathbf p=\mathbf q$ in this case.
     
      Comparing the coefficient of $y_{11}y_{1n}$ in the image of $x_{1n}x_{11}=p_{1n}x_{11}x_{1n}$ gives $q_{1n}=p_{1n}$. The corner subalgebras therefore have the same defining relations. 
      Identifying corresponding corner generators, the corner projection becomes a graded automorphism that preserves the subspaces $\kk x_{1n}$ and $\kk x_{n1}$. 
      Lemma~\ref{lem:two-fixed} shows that this automorphism is diagonal. Lemmas~\ref{lem:corner-lift} and~\ref{lem:corner-rigidity} then give $\mathbf p=\mathbf q$. The original parameter matrices therefore satisfy one of conditions~(1)--(4), by \eqref{eq:transpose-iso} and \eqref{eq:reversal-iso}.
\end{proof}

\section{From graded to full automorphisms}\label{sec:full}

We first show that every automorphism of $A$ is a product of a unipotent automorphism and a graded automorphism. Here we only assume $\lambda\ne1$. Let
\[
    {\rm UAut}(A)=\{\phi\in\Aut(A):\phi(x_{ij})-x_{ij}\in A_{\ge2}\text{ for all }i,j\}.
\]
Since $A$ is generated in degree one, this is the group of unipotent automorphisms defined in \cite[Definition~4.2]{GY2015}.

For each standard generator $x_{ij}$, we can choose another standard generator $x_{st}$ such that $x_{ij}x_{st}=c x_{st}x_{ij}$ with $c\ne1$. If $i\ne j$, take $x_{ii}$ when $p_{ij}\ne1$. When $p_{ij}=1$, take $x_{jj}$; then $c$ is $\lambda$ or $\lambda^{-1}$.
For a diagonal generator $x_{ii}$, choose $r\ne i$. Take $x_{ir}$ when $p_{ir}\ne1$. When $p_{ir}=1$, take $x_{ri}$; again, $c$ is $\lambda$ or $\lambda^{-1}$. Since $\lambda\ne1$, all these choices give $c\ne1$.
By \cite[Proposition~4.2]{LL2007}, every $\phi\in\Aut(A)$ therefore satisfies
\[
    \phi(A_d)\subseteq A_{\ge d},\quad d\ge0.
\]
By \cite[Lemma~4.8]{GY2015}, every automorphism of $A$ is a product of a unipotent automorphism and a graded automorphism. The subgroup $\Autgr(A)$ normalizes ${\rm UAut}(A)$, and their intersection is trivial. Hence
\begin{equation}\label{eq:full-split}
    \Aut(A)={\rm UAut}(A)\rtimes\Autgr(A).
\end{equation}

Goodearl and Yakimov proved a rigidity theorem for unipotent automorphisms of symmetric saturated CGL extensions. They applied it to give a new proof of the classification of automorphism groups of single-parameter quantum matrix algebras when $q$ is not a root of unity \cite[Theorem~4.3 and Section~4C]{GY2015}.

When $\lambda$ is not a root of unity, $A$ is a symmetric CGL extension with the generators $x_{ij}$ adjoined in lexicographic order. Its standard grading is induced by the natural torus action; see \cite[Section~5.6]{GL2000} and \cite[Example~5.8]{GY2018}. Its saturation condition can be written in terms of the parameters as follows.
Put
\[
    c_{(l,m),(i,j)}=
    \begin{cases}
        p_{li}p_{jm},&l\ge i,\ m> j,\\
        \lambda p_{li}p_{jm},&l>i,\ m\le j,
    \end{cases}
\]
and set $c_{(i,j),(l,m)}=c_{(l,m),(i,j)}^{-1}$ and $c_{(i,j),(i,j)}=1$. These are the commutation coefficients in this presentation. Let
\[
    L=\left\{a=(a_{ij})\in\mathbb Z^{n^2}:\prod_{i,j=1}^n c_{(l,m),(i,j)}^{a_{ij}}=1\text{ for all }l,m\right\}.
\]
By \cite[Definition~2.9]{GY2015}, this presentation is saturated if and only if $ra\in L$ implies $a\in L$ for every $a\in\mathbb Z^{n^2}$ and every positive integer $r$.

The following proposition applies the rigidity theorem to $A$.

\begin{proposition}\label{prop:full-saturated}
      Suppose that $\lambda$ is not a root of unity and that the CGL extension presentation of $A$ is saturated. Then every automorphism of $A$ is graded.
\end{proposition}

\begin{proof}
      By \cite[Proposition~2.7]{GY2015} and the defining relations, the only standard generators that are prime are $x_{1n}$ and $x_{n1}$. Both occur on the right side of
      \[
            x_{nn}x_{11}-x_{11}x_{nn}=(\lambda-1)p_{n1}x_{1n}x_{n1}.
      \]
      Since $x_{11}$ and $x_{nn}$ are not prime, the definition of the core in \cite[Section~4A]{GY2015} gives $C(A)=A$. Hence \cite[Theorem~4.3]{GY2015} gives ${\rm UAut}(A)=\{{\rm id}_A\}$. Equation~\eqref{eq:full-split} now gives $\Aut(A)=\Autgr(A)$.
\end{proof}

Put $\Gamma=\langle\lambda,p_{ij}:1\le j<i\le n\rangle\subseteq\kk^*$. If $\Gamma$ is torsion-free, then $\lambda$ is not a root of unity because $\lambda\ne1$, and the CGL extension presentation of $A$ is saturated; see \cite[Definition~2.9]{GY2015}. Proposition~\ref{prop:full-saturated} and Theorem~\ref{thm:main} therefore give the following corollary.
\begin{corollary}\label{cor:full}
    Suppose $\lambda\ne1$ and $\Gamma$ is torsion-free. Then every automorphism of $A$ is graded, and
    \[
        \Aut(A)=\Autgr(A)=
        \begin{cases}
            \HH\rtimes\langle\tau\rangle, & \lambda p_{ij}^{2}=1\quad\text{for all }i>j,\\
            \HH, & \text{otherwise}.
        \end{cases}
    \]
\end{corollary}
\begin{remark}
      Recall that setting $\lambda=q^{-2}$ and $p_{ij}=q$ for all $i>j$ gives the single-parameter quantum matrix algebra $\mathcal O_q(M_n(\kk))$. 
      Under the hypotheses of Corollary~\ref{cor:full}, the transposition case occurs only in the single-parameter family. Indeed, suppose $\lambda p_{ij}^2=1$ for all $i>j$, and put $q=p_{21}$. Then
      \[
            (p_{ij}q^{-1})^2=1,\quad i>j.
      \]
      Since $p_{ij}q^{-1}\in\Gamma$ and $\Gamma$ is torsion-free, we have $p_{ij}=q$ for all $i>j$. Thus $\lambda=q^{-2}$ and $A=\mathcal O_q(M_n(\kk))$.
\end{remark}

\section{Specializations and examples}\label{sec:specializations}

When $q$ is not a root of unity, the following graded classification follows from the known classification of the full automorphism group \cite{AC1992,Yakimov2013}. Theorem~\ref{thm:main} also gives this description when $q$ is a root of unity of order greater than two.

\begin{corollary}\label{cor:single-parameter}
	Let $q\in\kk^*$ with $q^2\ne1$. Then
	\[
		\Autgr(\mathcal O_q(M_n(\kk)))=\HH\rtimes\langle\tau\rangle.
	\]
\end{corollary}
\begin{proof}
	The parameters satisfy $\lambda p_{ij}^2=1$ for all $i>j$. Apply Theorem~\ref{thm:main}.
\end{proof}


\begin{example}[Transposition]
	Assume $\operatorname{char}\kk\ne2$, let $q\in\kk^*$ satisfy $q^2\ne1$, and take $n=3$, $\lambda=q^{-2}$ and
	\[
		\mathbf p=
		\begin{pmatrix}
			1&q^{-1}&-q^{-1}\\
			q&1&q^{-1}\\
			-q&q&1
		\end{pmatrix}.
	\]
	Then $\lambda p_{ij}^2=1$ for all $i>j$, so Theorem~\ref{thm:main} gives $\Autgr(A)=\HH\rtimes\langle\tau\rangle$. 
      The lower-triangular parameters are not all equal, and the group they generate contains $-1$. 
\end{example}

\begin{example}[Reversal]
      Assume $\operatorname{char}\kk\ne2$, let $\lambda=-1$, and choose $u\in\kk^*$ and $\varepsilon\in\{1,-1\}$. Set
	\[
		\mathbf p=
		\begin{pmatrix}
			1&u^{-1}&\varepsilon\\
			u&1&u\\
			\varepsilon&u^{-1}&1
		\end{pmatrix}.
	\]
	Then $p_{ij}=p_{4-i,4-j}$ for all $i,j$, while $\lambda p_{31}^2=-1\ne1$. Thus $\Autgr(A)=\HH\rtimes\langle\sigma\rangle$. The involution fixes $x_{22}$ and interchanges opposite positions.
\end{example}

\begin{remark}[The specialization $\lambda=1$]\label{rem:affine}
    When $\lambda=1$, $A$ is the quantum affine space $\mathcal O_{\mathbf Q}(\kk^{n^2})$ with
    \[
            \mathbf Q=\mathbf p^{\mathsf T}\otimes\mathbf p,\quad Q_{(i,j),(l,m)}=p_{li}p_{jm},
    \]
    with the convention $x_{lm}x_{ij}=Q_{(i,j),(l,m)}x_{ij}x_{lm}$. 
    Its graded automorphism group is given by the known classification \cite[Theorem~4.5]{Jin2025}; see also \cite{JS2026}. 
    In particular, if all $p_{ij}=1$, then $A$ is a polynomial algebra in $n^2$ variables and $\Autgr(A)=\GL_{n^2}(\kk)$.
\end{remark}

\section*{Declaration of generative AI and AI-assisted technologies in the manuscript preparation process}
Statement: 
During the preparation of this work, the author used OpenAI's ChatGPT to assist with language editing, manuscript organization, reference checking, and LaTeX formatting.
After using this tool, the author reviewed and edited the content as needed and takes full responsibility for the content of the published article.

\end{document}